\documentclass[oneside,reqno,english]{amsart}
\usepackage[T1]{fontenc}
\usepackage[utf8]{inputenc}
\usepackage{xcolor}
\usepackage{babel}
\usepackage{prettyref}
\usepackage{amstext}
\usepackage{amsthm}
\usepackage{amssymb}
\usepackage[pdfusetitle,
 bookmarks=true,bookmarksnumbered=false,bookmarksopen=false,
 breaklinks=false,pdfborder={0 0 0},pdfborderstyle={},backref=false,colorlinks=false]
 {hyperref}
\hypersetup{
 colorlinks=true,citecolor=blue,linkcolor=blue,linktocpage=true}

\makeatletter
\numberwithin{equation}{section}
\numberwithin{figure}{section}

\usepackage{prettyref}

\newrefformat{cor}{Corollary~\ref{#1}}
\newrefformat{subsec}{Section~\ref{#1}}
\newrefformat{lem}{Lemma~\ref{#1}}
\newrefformat{thm}{Theorem~\ref{#1}}
\newrefformat{sec}{Section~\ref{#1}}
\newrefformat{chap}{Chapter~\ref{#1}}
\newrefformat{prop}{Proposition~\ref{#1}}
\newrefformat{exa}{Example~\ref{#1}}
\newrefformat{tab}{Table~\ref{#1}}
\newrefformat{rem}{Remark~\ref{#1}}
\newrefformat{def}{Definition~\ref{#1}}
\newrefformat{fig}{Figure~\ref{#1}}
\newrefformat{claim}{Claim~\ref{#1}}
\newrefformat{assu}{Assumption~\ref{#1}}

\makeatother

\theoremstyle{plain}
\newtheorem{thm}{\protect\theoremname}[section]
\newtheorem{prop}[thm]{\protect\propositionname}
\theoremstyle{remark}
\newtheorem{rem}[thm]{\protect\remarkname}
\theoremstyle{plain}
\newtheorem{cor}[thm]{\protect\corollaryname}
\theoremstyle{definition}
\newtheorem{example}[thm]{\protect\examplename}
\theoremstyle{plain}
\newtheorem{lem}[thm]{\protect\lemmaname}
\providecommand{\corollaryname}{Corollary}
\providecommand{\examplename}{Example}
\providecommand{\lemmaname}{Lemma}
\providecommand{\propositionname}{Proposition}
\providecommand{\remarkname}{Remark}
\providecommand{\theoremname}{Theorem}

\begin{document}
\title[Subinvariant kernel dynamics]{Subinvariant kernel dynamics}
\begin{abstract}
We study positive definite kernels pulled back along a finite family
of self-maps under a subinvariance inequality for the associated branching
operator. Iteration produces an increasing kernel tower with defect
kernels. Under diagonal boundedness, the tower has a smallest invariant
majorant, with a canonical defect space realization and an explicit
diagonal harmonic envelope governing finiteness versus blow-up. We
also give probabilistic and boundary representations: a Gaussian martingale
model whose quadratic variation is the defect sequence, and canonical
Doob path measures with a boundary feature model for the normalized
defects.
\end{abstract}

\author{James Tian}
\address{Mathematical Reviews, 535 W. William St, Suite 210, Ann Arbor, MI
48103, USA}
\email{james.ftian@gmail.com}

\maketitle
\tableofcontents{}

\section{Introduction}\label{sec:1}

We consider a simple but robust situation in which a positive definite
kernel $K:X\times X\rightarrow\mathbb{C}$ on a set $X$ is repeatedly
pulled back along a finite family of maps $\phi_{1},\dots,\phi_{m}:X\to X$.
The basic object is the associated ``branching'' operator $L$ acting
on kernels, 
\[
\left(LJ\right)\left(s,t\right):=\sum^{m}_{i=1}J\left(\phi_{i}\left(s\right),\phi_{i}\left(t\right)\right),
\]
and we assume throughout the subinvariance inequality $LK\geq K$
in the positive-definite order. Informally, this condition says that
the kernel does not lose positive mass when transported along the
$\phi$-tree. It is reminiscent of subharmonicity and of inequalities
of the form $\Phi\left(A\right)\geq A$ for completely positive maps
on operator algebras \cite{MR1976867}, and it appears implicitly
in a variety of settings: iterated function systems and self-similar
structures, kernel-based models for non-reversible Markov dynamics,
and contractive interpolation problems where a given kernel sits under
a more symmetric or invariant model kernel, cf. \cite{MR3616205,MR1882259,MR1976867,MR3526117}.
From a different but complementary viewpoint, operators obtained by
transporting functions or observables along dynamics (Koopman and
Perron-Frobenius type operators) have recently been analyzed within
reproducing kernel Hilbert spaces (RKHS), where the kernel plays the
role of a state space metric and transfer operators admit spectral
and approximation theory in $\mathcal{H}_{K}$ \cite{MR4054854};
related constructions produce entire families of RKHS adapted to ensemble
dynamics \cite{MR4682643}.

Starting from $K$ and $\left(\phi_{i}\right)$ (\prettyref{sec:2}),
the dynamics of $L$ generates a tower of kernels 
\[
K_{0}:=K,\qquad K_{n+1}:=LK_{n}\quad\left(n\geq0\right),
\]
together with their defects $D_{n}:=K_{n+1}-K_{n}$. One may view
$K_{n}$ as the $n$-step kernel obtained by following all length-$n$
branches in the rooted $\phi$-tree, while $D_{n}$ records the new
positive mass injected between levels $n$ and $n+1$. The first goal
of the paper is to understand the ``invariant completion'' of this
process: under what conditions does the increasing tower $\left(K_{n}\right)$
converge to a limiting kernel $K_{\infty}$, what structure does $K_{\infty}$
inherit from $L$, and how does the original kernel $K$ sit inside
this completion? Throughout, positive definiteness is understood in
the classical sense of Aronszajn \cite{MR51437} (see also \cite{MR3526117,MR3560890,MR2239907}).
A natural finiteness requirement is the diagonal boundedness condition
\[
\sup_{n\geq0}K_{n}\left(s,s\right)<\infty\quad\text{for all }s\in X.
\]
This hypothesis is deliberately weak: we do not assume any topology
or measure on $X$, nor any compactness or spectral input for the
maps $\phi_{i}$. Our first main result shows that, under this assumption,
the tower admits a canonical $L$-invariant completion. The kernels
$K_{n}$ increase pointwise and converge to a positive definite kernel
$K_{\infty}$ satisfying 
\[
K_{\infty}\geq K,\qquad LK_{\infty}=K_{\infty},
\]
and $K_{\infty}$ is minimal among all $L$-invariant kernels majorizing
$K$. Moreover, the defect tower admits an orthogonal realization:
one can represent $K_{\infty}$ on an ambient Hilbert space obtained
by adjoining defect spaces level by level, so that $K$ is recovered
from $K_{\infty}$ by a canonical compression. In particular, the
subinvariance dynamics induces a concrete Radon-Nikodym type representation
of $K$ inside its own invariant completion, in a spirit close to
Sz.-Nagy-Foiaş defect decompositions and model theory for contractions
\cite{MR2760647}, but carried out at the level of kernels and without
assuming an operator on a pre-existing Hilbert space. The $L$-invariant
completion also yields a natural multivariable row-isometry associated
to the maps $\phi_{i}$, placing the construction within the scope
of multivariable dilation theory \cite{MR1668582}. In particular,
the operator-theoretic structure of row isometries and their Wold-type
decompositions \cite{MR2416726,MR4651635} provides a useful organizing
analogy for the orthogonal defect splitting produced here.

The second theme of the paper is to understand when this completion
is genuinely finite and where it breaks down (\prettyref{sec:3}).
The key observation is that the diagonal dynamics decouples as a purely
combinatorial process on the $\phi$-tree. Writing 
\[
u_{n}\left(s\right):=K_{n}\left(s,s\right),\qquad s\in X,
\]
one obtains a branching operator $P$ on nonnegative functions, 
\[
\left(Pu\right)\left(s\right):=\sum^{m}_{i=1}u\left(\phi_{i}\left(s\right)\right),
\]
so that $u_{n+1}=Pu_{n}$ and hence $u_{n}=P^{n}u_{0}$. We show that
the pointwise growth of $\left(u_{n}\right)$ is governed by a potential-theoretic
picture: the function 
\[
h_{\infty}\left(s\right):=\sup_{n\geq0}u_{n}\left(s\right)
\]
is the minimal $P$-harmonic majorant of $u_{0}\left(s\right)=K\left(s,s\right)$.
This leads to a forward-invariant finiteness region 
\[
X_{\mathrm{fin}}:=\left\{ s\in X:\ h_{\infty}\left(s\right)<\infty\right\} ,
\]
on which the invariant completion $K_{\infty}$ is well-defined and
finite on the diagonal, while outside this region the diagonal diverges
monotonically along the tower. Beyond this characterization, we give
usable tree-based criteria---Lyapunov-type decay conditions and complementary
non-decay mechanisms based on branch counting---which are similar
to classical potential theory on trees and branching Markov chains,
but are formulated here for the diagonal entries of a kernel tower
rather than for scalar harmonic functions \cite{MR3616205,MR2548569,MR1324344}.
Recent work on boundary representations for harmonic and polyharmonic
functions on trees \cite{MR4031266,MR4412972} and on Martin boundaries
for branching-type dynamics \cite{MR4798615} provides additional
context for the role of harmonic majorants and finiteness regions
in tree-indexed growth problems.

A third component of the work is a probabilistic reinterpretation
of the defect tower and of the compression (\prettyref{sec:4}). We
build a Gaussian field whose level truncations have covariance kernels
$K_{N}$ and form a martingale with orthogonal increments; the predictable
quadratic variation of this martingale is  the defect tower $\left(D_{n}\right)$.
On finite sets this martingale converges precisely when the Gram matrices
of $K_{N}$ remain uniformly bounded, and the limiting covariance
is $K_{\infty}$. At the level of the ambient Hilbert space, the compression
operator becomes a literal Radon-Nikodym derivative between two Gaussian
covariance structures, one with kernel $K_{\infty}$ and one with
kernel $K$. This point of view is in line with the well-known correspondence
between Gaussian processes and RKHS \cite{MR2514435,MR747302,MR3024389,MR2239907},
but here the Gaussian structure is tailored to the specific defect
decomposition induced by the subinvariant dynamics. Boundary processes
and jump dynamics on random trees \cite{MR4046511} offer a complementary
probabilistic perspective on how tree boundaries encode limiting behavior
of tree-indexed evolutions.

The last part (\prettyref{sec:5}) is a boundary/pathspace organization
of the completion. We make the implicit tree structure of the word
dynamics explicit by passing to the symbolic boundary $\Omega=\left\{ 1,\dots,m\right\} ^{\mathbb{N}}$
and by extracting, from the minimal harmonic majorant $h_{\infty}$,
a canonical family of Doob-transformed Markov measures $\mu_{s}$
on $\Omega$. This produces a normalized “Doob” version of the branching
dynamics and yields boundary cylinder expansions for the normalized
kernel iterates and for the defect tower, which can then be realized
as an $L^{2}$-boundary Gram kernel via an explicit boundary feature
map. This places the present kernel subinvariance problem in the orbit
of pathspace measure and boundary representation constructions for
discrete branching/graph systems \cite{MR4388381,MR4730766,MR4412972},
while keeping the output at the level of kernels/RKHS (rather than
only harmonic functions or operator-algebraic data). 

The overall idea is that the inequality $LK\geq K$, when iterated
along a finite branching system of maps, forces a rich and fairly
rigid structure on the invariant completion and on the diagonal growth
dynamics. The resulting picture has points of contact with dilation
and model theory in operator theory \cite{MR2760647,MR1668582,MR2416726,MR4651635},
with potential theory on trees and branching processes \cite{MR3616205,MR2548569,MR1324344,MR4031266,MR4412972,MR4798615},
and with Gaussian/RKHS methods in probability and statistics \cite{MR2239907,MR2514435,MR3024389}.
It also aligns with the broader scope of kernel-based operator methods
for dynamics, where transfer operators are studied via RKHS structure
\cite{MR4054854,MR4682643}. At the same time, the setting here is
combinatorial and kernel-theoretic: no measure, topology, or ambient
operator algebra is required.

\section{Subinvariance and invariant majorants}\label{sec:2}

In this section, we develop the basic subinvariance theory for the
branching operator $L$ acting on positive definite kernels. 

Let $X$ be a set, $K:X\times X\to\mathbb{C}$ a positive definite
(p.d.) kernel, and $\phi_{1},\dots,\phi_{m}:X\to X$ maps on $X$.
We use the standard convention $K_{s}\left(\cdot\right):=K\left(\cdot,s\right)$,
so that $K\left(s,t\right)=\left\langle K_{s},K_{t}\right\rangle _{\mathcal{H}_{K}}$
and $f\left(s\right)=\left\langle K_{s},f\right\rangle _{\mathcal{H}_{K}}$,
for all $s\in X$ and $f\in\mathcal{H}_{K}$, where $\mathcal{H}_{K}$
is the reproducing kernel Hilbert space (RKHS) of $K$. Throughout,
all inner products are linear in the second variable. 

Let $\mathcal{K}\left(X\right)$ denote the set of all p.d. kernels
on $X$, and write $J_{1}\ge J_{2}$ if $J_{1}-J_{2}\in\mathcal{K}\left(X\right)$.

Define linear operator $L$ acting on $\mathcal{K}\left(X\right)$
by 
\[
\left(LJ\right)\left(s,t\right):=\sum^{m}_{i=1}J\left(\phi_{i}\left(s\right),\phi_{i}\left(t\right)\right).
\]
For any $J\in\mathcal{K}\left(X\right)$, each pullback kernel 
\[
\left(s,t\right)\mapsto J\left(\phi_{i}\left(s\right),\phi_{i}\left(t\right)\right)
\]
is p.d., hence $LJ\in\mathcal{K}\left(X\right)$. In particular, $L$
is order-preserving: $J_{1}\ge J_{2}$ implies $LJ_{1}\ge LJ_{2}$. 

Assume the subinvariance inequality 
\begin{equation}
LK\ge K.\label{eq:2-1}
\end{equation}
Equivalently, the defect kernel 
\[
D:=LK-K
\]
is positive definite, i.e., $D\in\mathcal{K}\left(X\right)$. 

For all $n\in\mathbb{N}_{0}:=\left\{ 0,1,2,\dots\right\} $, define
the tower of kernels 
\[
K_{n+1}:=LK_{n},\quad K_{0}:=K,
\]
and the defect iterates 
\[
D_{n}:=K_{n+1}-K_{n},\quad D_{0}:=D.
\]

Finally, assume the following finiteness condition: 
\begin{equation}
\sup_{n\ge0}K_{n}\left(s,s\right)<\infty,\quad\forall s\in X.\label{eq:2-2}
\end{equation}

\begin{thm}
\label{thm:2-1}Under the setting above:
\begin{enumerate}
\item For all $s,t\in X$ the limit 
\begin{equation}
K_{\infty}\left(s,t\right):=\lim_{n\to\infty}K_{n}\left(s,t\right)
\end{equation}
exists (finite), and $K_{\infty}$ belongs to $\mathcal{K}\left(X\right)$
satisfying 
\[
K_{\infty}\ge K,\quad LK_{\infty}=K_{\infty}.
\]
Moreover, $K_{\infty}$ is the minimal $L$-invariant majorant of
$K$. That is, if $J\in\mathcal{K}\left(X\right)$, $J\ge K$, and
$LJ=J$, then $J\ge K_{\infty}$.
\item The defects satisfy $D_{n}=L^{n}D$, and for every $N\ge1$, 
\begin{equation}
K_{N}=K+\sum^{N-1}_{n=0}D_{n}.\label{eq:2-3}
\end{equation}
In particular, 
\begin{equation}
K_{\infty}=K+\sum^{\infty}_{n=0}D_{n},\label{eq:2-4}
\end{equation}
where the series converges in the sense of Gram matrices on each finite
subset of $X$.
\item Let $\mathcal{H}_{D_{n}}$ be the RKHS of $D_{n}$, and set 
\begin{equation}
\mathcal{E}:=\mathcal{H}_{K}\oplus\bigoplus_{n\ge0}\mathcal{H}_{D_{n}}.\label{eq:2-5}
\end{equation}
For each $s\in X$ define 
\[
v\left(s\right):=K_{s}\oplus\bigoplus_{n\ge0}\left(D_{n}\right)_{s}\in\mathcal{E}.
\]
Then $v\left(s\right)$ is well-defined and 
\[
K_{\infty}\left(s,t\right)=\left\langle v\left(s\right),v\left(t\right)\right\rangle _{\mathcal{E}}.
\]
Let $P_{0}:\mathcal{E}\to\mathcal{H}_{K}$ be the coordinate projection
onto the first summand, and define $A:=P^{*}_{0}P_{0}$ (a positive
contraction on $\mathcal{E}$). Then 
\[
K\left(s,t\right)=\left\langle v\left(s\right),Av\left(t\right)\right\rangle _{\mathcal{E}}.
\]
Thus, the subinvariance dynamics canonically produces a Radon-Nikodym
compression representation of $K$ inside the invariant completion
$K_{\infty}$, in the sense that $K_{\infty}\left(s,t\right)=\left\langle v\left(s\right),v\left(t\right)\right\rangle _{\mathcal{E}}$
and $K\left(s,t\right)=\left\langle v\left(s\right),Av\left(t\right)\right\rangle _{\mathcal{E}}$
for all $s,t\in X$. 
\end{enumerate}
\end{thm}

\begin{proof}
Assumption \eqref{eq:2-1} implies 
\[
K_{1}-K_{0}=LK-K=D_{0}\ge0.
\]
Inductively, if $K_{n}\ge K_{n-1}$, then $K_{n+1}=LK_{n}\ge LK_{n-1}=K_{n}$.
Hence $\left\{ K_{n}\right\} $ is an increasing sequence. In particular,
each difference 
\[
D_{n}:=K_{n+1}-K_{n}
\]
is positive definite. Moreover, from $K_{n+1}=LK_{n}$ we compute
\[
D_{n}=K_{n+1}-K_{n}=LK_{n}-LK_{n-1}=L\left(K_{n}-K_{n-1}\right)=LD_{n-1}.
\]
Thus $D_{n}=L^{n}D$ for all $n\ge0$. Thus, for $N\ge1$, 
\[
K_{N}-K=\sum^{N-1}_{n=0}\left(K_{n+1}-K_{n}\right)=\sum^{N-1}_{n=0}D_{n},
\]
which proves the finite-level energy identity \eqref{eq:2-3}.

Fix a finite set $F=\left\{ x_{1},\dots,x_{r}\right\} \subset X$.
Consider the Gram matrices 
\[
G_{n}:=\left[K_{n}\left(x_{a},x_{b}\right)\right]^{r}_{a,b=1}.
\]
Each $G_{n}$ is positive semidefinite, and $G_{n+1}-G_{n}=\left[D_{n}\left(x_{a},x_{b}\right)\right]$
is positive semidefinite; hence $G_{n}$ is increasing in the order
on Hermitian matrices. By \eqref{eq:2-2}, the diagonal entries $\left(G_{n}\right)_{aa}=K_{n}\left(x_{a},x_{a}\right)$
are uniformly bounded in $n$ for each $a$, and therefore 
\[
{\rm tr}\left(G_{n}\right)=\sum^{r}_{a=1}K_{n}\left(x_{a},x_{a}\right)\le\sum^{r}_{a=1}\sup_{n\ge0}K_{n}\left(x_{a},x_{a}\right)<\infty.
\]
Since $G_{n}\ge0$, one has $\left\Vert G_{n}\right\Vert \le{\rm tr}\left(G_{n}\right)$,
so $\sup_{n}\left\Vert G_{n}\right\Vert <\infty$. Thus $\left\{ G_{n}\right\} $
is increasing and uniformly bounded in operator norm, and hence converges
(in operator norm) to a positive semidefinite matrix 
\[
G_{\infty}:=\lim_{n\to\infty}G_{n}.
\]
In particular, for each $1\le a,b\le r$, 
\[
\left(G_{\infty}\right)_{ab}=\lim_{n\to\infty}K_{n}\left(x_{a},x_{b}\right).
\]
Since this holds for every finite $F$, the pointwise limit 
\[
K_{\infty}\left(s,t\right):=\lim_{n\to\infty}K_{n}\left(s,t\right)
\]
defines a positive definite kernel. The monotonicity gives $K_{\infty}\ge K$
immediately.

To see $LK_{\infty}=K_{\infty}$, fix $s,t\in X$ and compute 
\[
\left(LK_{\infty}\right)\left(s,t\right)=\sum^{m}_{i=1}K_{\infty}\left(\phi_{i}\left(s\right),\phi_{i}\left(t\right)\right)=\sum^{m}_{i=1}\lim_{n\to\infty}K_{n}\left(\phi_{i}\left(s\right),\phi_{i}\left(t\right)\right).
\]
Since $m$ is finite, we may pass the limit through the sum, obtaining
\begin{align*}
\left(LK_{\infty}\right)\left(s,t\right) & =\lim_{n\to\infty}\sum^{m}_{i=1}K_{n}\left(\phi_{i}\left(s\right),\phi_{i}\left(t\right)\right)\\
 & =\lim_{n\to\infty}\left(LK_{n}\right)\left(s,t\right)=\lim_{n\to\infty}K_{n+1}\left(s,t\right)=K_{\infty}\left(s,t\right).
\end{align*}
For minimality: if $J\ge K$ and $LJ=J$, then by positivity of $L$,
\[
J=L^{n}J\ge L^{n}K=K_{n}\quad\text{for all }n.
\]
Taking $n\to\infty$ yields $J\ge K_{\infty}$. This completes (1).

Finally, the identity \eqref{eq:2-4} holds in the sense that for
each finite $F$, the Gram matrices satisfy 
\[
\left[K_{\infty}\left(x_{a},x_{b}\right)\right]_{a,b}=\left[K\left(x_{a},x_{b}\right)\right]_{a,b}+\sum^{\infty}_{n=0}\left[D_{n}\left(x_{a},x_{b}\right)\right]_{a,b},
\]
with convergence in the order on Hermitian matrices because the partial
sums are $G_{N}-G_{0}$ and increase to $G_{\infty}-G_{0}$. This
is the Gram-matrix form of (2).

Define the Hilbert space 
\[
\mathcal{E}:=\mathcal{H}_{K}\oplus\bigoplus_{n\ge0}\mathcal{H}_{D_{n}}.
\]
Fix $s\in X$. We claim that 
\[
\sum_{n\ge0}\left\Vert \left(D_{n}\right)_{s}\right\Vert ^{2}_{\mathcal{H}_{D_{n}}}<\infty,
\]
so that $v\left(s\right)=K_{s}\oplus\bigoplus_{n\ge0}\left(D_{n}\right)_{s}$
is a well-defined element of $\mathcal{E}$. Indeed, in any RKHS $\mathcal{H}_{J}$
one has $\left\Vert J_{s}\right\Vert ^{2}_{\mathcal{H}_{J}}=J\left(s,s\right)$.
Therefore 
\begin{align*}
\sum^{N-1}_{n=0}\left\Vert \left(D_{n}\right)_{s}\right\Vert ^{2}_{\mathcal{H}_{D_{n}}} & =\sum^{N-1}_{n=0}D_{n}\left(s,s\right)\\
 & =\sum^{N-1}_{n=0}\left(K_{n+1}\left(s,s\right)-K_{n}\left(s,s\right)\right)=K_{N}\left(s,s\right)-K\left(s,s\right).
\end{align*}
Letting $N\to\infty$ and using \eqref{eq:2-2} gives finiteness:
\[
\sum_{n\ge0}\left\Vert \left(D_{n}\right)_{s}\right\Vert ^{2}_{\mathcal{H}_{D_{n}}}=K_{\infty}\left(s,s\right)-K\left(s,s\right)<\infty.
\]
Hence $v\left(s\right)\in\mathcal{E}$ for every $s$. Now fix $s,t\in X$.
For each $N$, 
\[
\left\langle K_{s}\oplus\bigoplus^{N-1}_{n=0}\left(D_{n}\right)_{s},\ K_{t}\oplus\bigoplus^{N-1}_{n=0}\left(D_{n}\right)_{t}\right\rangle _{\mathcal{E}}=K\left(s,t\right)+\sum^{N-1}_{n=0}D_{n}\left(s,t\right)=K_{N}\left(s,t\right),
\]
using the RKHS identities $\left\langle K_{s},K_{t}\right\rangle _{\mathcal{H}_{K}}=K\left(s,t\right)$
and $\left\langle \left(D_{n}\right)_{s},\left(D_{n}\right)_{t}\right\rangle _{\mathcal{H}_{D_{n}}}=D_{n}\left(s,t\right)$,
together with the telescoping identity. Let $N\to\infty$. Then $K_{N}\left(s,t\right)\to K_{\infty}\left(s,t\right)$,
and the partial sums in $\mathcal{E}$ converge in norm to $v\left(s\right)$
and $v\left(t\right)$. Therefore the inner products converge, yielding
\[
\left\langle v\left(s\right),v\left(t\right)\right\rangle _{\mathcal{E}}=\lim_{N\to\infty}K_{N}\left(s,t\right)=K_{\infty}\left(s,t\right).
\]
This proves the feature representation in (3). Let $P_{0}:\mathcal{E}\to\mathcal{H}_{K}$
be the coordinate projection onto the first summand, and $A:=P^{*}_{0}P_{0}$.
Then $A$ is a positive contraction on $\mathcal{E}$, and 
\[
\left\langle v\left(s\right),Av\left(t\right)\right\rangle _{\mathcal{E}}=\left\langle P_{0}v\left(s\right),P_{0}v\left(t\right)\right\rangle _{\mathcal{H}_{K}}=\left\langle K_{s},K_{t}\right\rangle _{\mathcal{H}_{K}}=K\left(s,t\right),
\]
as claimed.
\end{proof}

\begin{prop}
\label{prop:2-2}Keep the setting and define $K_{\infty}$ as above.
Let $\mathcal{H}_{K_{\infty}}$ be the RKHS of $K_{\infty}$. On $\mathrm{span}\left\{ \left(K_{\infty}\right)_{s}:s\in X\right\} $
define 
\[
V:\mathrm{span}\left\{ \left(K_{\infty}\right)_{s}:s\in X\right\} \to\mathcal{H}^{\oplus m}_{K_{\infty}},\qquad V\left(K_{\infty}\right)_{s}:=\bigoplus^{m}_{i=1}\left(K_{\infty}\right)_{\phi_{i}\left(s\right)}.
\]
Then $V$ extends to an isometry, and writing $V=\bigoplus^{m}_{i=1}V_{i}$
with $V_{i}\in B\left(\mathcal{H}_{K_{\infty}}\right)$ we have 
\[
\sum^{m}_{i=1}V^{*}_{i}V_{i}=I.
\]
Moreover, there exists a positive contraction $A\in B\left(\mathcal{H}_{K_{\infty}}\right)$
such that 
\begin{equation}
K\left(s,t\right)=\left\langle \left(K_{\infty}\right)_{s},A\left(K_{\infty}\right)_{t}\right\rangle _{\mathcal{H}_{K_{\infty}}}\quad\left(s,t\in X\right),\label{eq:2-6}
\end{equation}
and $LK\ge K$ is equivalent to 
\[
\sum^{m}_{i=1}V^{*}_{i}AV_{i}\ge A
\]
(verified on $\mathrm{span}\left\{ \left(K_{\infty}\right)_{s}:s\in X\right\} $,
hence in the form sense on $\mathcal{H}_{K_{\infty}}$).
\end{prop}

\begin{proof}
From $LK_{\infty}=K_{\infty}$ we get, for $s,t\in X$, 
\begin{align*}
\left\langle V\left(K_{\infty}\right)_{s},V\left(K_{\infty}\right)_{t}\right\rangle _{\mathcal{H}^{\oplus m}_{K_{\infty}}} & =\sum^{m}_{i=1}K_{\infty}\left(\phi_{i}\left(s\right),\phi_{i}\left(t\right)\right)\\
 & =K_{\infty}\left(s,t\right)=\left\langle \left(K_{\infty}\right)_{s},\left(K_{\infty}\right)_{t}\right\rangle _{\mathcal{H}_{K_{\infty}}},
\end{align*}
so $V$ is isometric on the dense span of kernel sections and extends
to an isometry. Writing $V=\bigoplus_{i}V_{i}$ then yields 
\[
\left\langle f,g\right\rangle =\left\langle Vf,Vg\right\rangle =\sum_{i}\left\langle V_{i}f,V_{i}g\right\rangle =\left\langle f,\left(\sum\nolimits_{i}V^{*}_{i}V_{i}\right)g\right\rangle ,
\]
hence $\sum_{i}V^{*}_{i}V_{i}=I$. The operator $A$ is the compression
from \prettyref{thm:2-1}, realized on $\mathcal{H}_{K_{\infty}}$
by the identity 
\[
K\left(s,t\right)=\left\langle \left(K_{\infty}\right)_{s},A\left(K_{\infty}\right)_{t}\right\rangle .
\]
Finally, for $s,t\in X$, 
\begin{align*}
\sum_{i}K\left(\phi_{i}\left(s\right),\phi_{i}\left(t\right)\right) & =\sum_{i}\left\langle \left(K_{\infty}\right)_{\phi_{i}\left(s\right)},A\left(K_{\infty}\right)_{\phi_{i}\left(t\right)}\right\rangle \\
 & =\left\langle \left(K_{\infty}\right)_{s},\left(\sum\nolimits_{i}V^{*}_{i}AV_{i}\right)\left(K_{\infty}\right)_{t}\right\rangle ,
\end{align*}
so $LK\ge K$ iff $\sum_{i}V^{*}_{i}AV_{i}\ge A$ on $\mathrm{span}\left\{ \left(K_{\infty}\right)_{s}\right\} $,
hence in the form sense on $\mathcal{H}_{K_{\infty}}$.
\end{proof}

\begin{rem}
\prettyref{prop:2-2} shows that the Radon-Nikodym compression of
$K$ may be realized internally on $\mathcal{H}_{K_{\infty}}$ as
in \eqref{eq:2-6}. In particular, the invariant completion $K_{\infty}$
contains enough information to represent $K$ by a compression. The
auxiliary dilation space $\mathcal{E}$ from \eqref{eq:2-5} is nonetheless
useful. First, it makes the compression operator completely explicit:
on $\mathcal{E}$ one has $A=P^{*}_{0}P_{0}$, where $P_{0}:\mathcal{E}\to\mathcal{H}_{K}$
is the coordinate projection. Second, $\mathcal{E}$ retains the defect
tower as an orthogonal direct sum, so the decomposition \eqref{eq:2-4}
is realized at the Hilbert space level; this is convenient for truncations
and estimates which track contributions of the individual defect kernels
$D_{n}$.

The term ``Radon-Nikodym'' is from the Radon-Nikodym (RN) theorem
for completely positive maps on operator algebras. If $\Phi,\Psi$
are normal completely positive maps on a $C^{*}$-algebra or von Neumann
algebra with $\Phi\leq\Psi$ in the usual order, then in a Stinespring
dilation $(\pi,H_{\Psi},V)$ for $\Psi$ there exists a positive contraction
$0\leq B\leq I_{H_{\Psi}}$ such that $B\in\pi\left(\mathcal{A}\right)'$
and 
\[
\Phi(x)=V^{*}B\pi\left(x\right)V,\qquad x\in\mathcal{A},
\]
see for instance \cite{MR932932,MR1668582,MR1976867} for versions
of this theorem. In that setting $B$ is the RN derivative of $\Phi$
with respect to $\Psi$, acting on the dilation space $H_{\Psi}$.
The present kernel construction is a direct analogue at the level
RKHS (\cite{MR51437,MR3560890,tian2025k}).
\end{rem}

\begin{cor}
\label{cor:2-3} Let $\mathcal{W}_{n}:=\left\{ 1,\dots,m\right\} ^{n}$
be the set of words of length $n$, with $\mathcal{W}_{0}:=\left\{ \emptyset\right\} $.
For $w=i_{1}\cdots i_{n}\in\mathcal{W}_{n}$ set 
\[
\phi_{w}:=\phi_{i_{1}}\circ\cdots\circ\phi_{i_{n}},\qquad\phi_{\emptyset}:={\rm id}_{X}.
\]
Then for every $J\in\mathcal{K}\left(X\right)$ and every $n\ge0$,
\[
\left(L^{n}J\right)\left(s,t\right)=\sum_{w\in\mathcal{W}_{n}}J\left(\phi_{w}\left(s\right),\phi_{w}\left(t\right)\right).
\]
In particular, 
\begin{align*}
K_{n}\left(s,t\right) & =\sum_{w\in\mathcal{W}_{n}}K\left(\phi_{w}\left(s\right),\phi_{w}\left(t\right)\right),\\
D_{n}\left(s,t\right) & =\sum_{w\in\mathcal{W}_{n}}D\left(\phi_{w}\left(s\right),\phi_{w}\left(t\right)\right).
\end{align*}
Moreover, for every $N\ge1$, 
\[
K_{N}\left(s,t\right)=K\left(s,t\right)+\sum^{N-1}_{n=0}\sum_{w\in\mathcal{W}_{n}}D\left(\phi_{w}\left(s\right),\phi_{w}\left(t\right)\right),
\]
and, under \eqref{eq:2-2}, 
\[
K_{\infty}\left(s,t\right)=K\left(s,t\right)+\sum^{\infty}_{n=0}\sum_{w\in\mathcal{W}_{n}}D\left(\phi_{w}\left(s\right),\phi_{w}\left(t\right)\right),
\]
with convergence in the Gram-matrix sense on each finite subset of
$X$.
\end{cor}

\begin{proof}
The formula for $L^{n}$ follows by induction on $n$: the case $n=0$
is trivial, and the step $n\to n+1$ is obtained by expanding $L^{n+1}J=L\left(L^{n}J\right)$
and regrouping the resulting sum over words of length $n+1$. The
identities for $K_{n}$ and $D_{n}$ follow from $K_{n}=L^{n}K$ and
$D_{n}=L^{n}D$. The finite and infinite expansions are the identities
from \eqref{eq:2-3}-\eqref{eq:2-4} with these formulas inserted. 
\end{proof}

We now include a concrete example illustrating the full setting above,
including strict subinvariance, the finiteness hypothesis \eqref{eq:2-2},
and a nontrivial defect tower. The example is based on word dynamics
on an $m$-ary tree, with $\phi_{i}$ given by prefixing. The kernel
is chosen so that $L$-invariance holds at the level of an explicit
majorant while the defect decays geometrically, leading to $K_{\infty}$
in closed form. 
\begin{example}
\label{exa:2-5} Fix an integer $m\ge2$ and let 
\[
X:=\mathcal{W}:=\bigcup_{n\ge0}\left\{ 1,\dots,m\right\} ^{n}
\]
be the set of all finite words over $\{1,\dots,m\}$, including the
empty word $\emptyset$. For $w\in\mathcal{W}$ write $\left|w\right|$
for the word length. Define maps $\phi_{i}:X\to X$ by prefixing:
\[
\phi_{i}\left(w\right):=iw,\qquad1\le i\le m.
\]
Define two positive definite kernels $J_{0},J_{1}:X\times X\to\mathbb{C}$
by 
\[
J_{0}\left(u,v\right):=m^{-\left|u\right|}\delta_{u,v},\qquad J_{1}\left(u,v\right):=m^{-\left(\left|u\right|+\left|v\right|\right)/2}.
\]
Then $J_{0}$ is positive definite (diagonal with nonnegative entries),
and $J_{1}$ is positive definite (rank-one), since $J_{1}\left(u,v\right)=f\left(u\right)\overline{f\left(v\right)}$
with $f\left(w\right)=m^{-\left|w\right|/2}$. Moreover, both are
$L$-invariant: 
\[
LJ_{0}=J_{0},\qquad LJ_{1}=J_{1}.
\]
Indeed, for $s,t\in X$, 
\[
\left(LJ_{0}\right)\left(s,t\right)=\sum^{m}_{i=1}m^{-\left|is\right|}\delta_{is,it}=\sum^{m}_{i=1}m^{-\left(\left|s\right|+1\right)}\delta_{s,t}=J_{0}\left(s,t\right),
\]
and 
\[
\left(LJ_{1}\right)\left(s,t\right)=\sum^{m}_{i=1}m^{-\left(\left|is\right|+\left|it\right|\right)/2}=\sum^{m}_{i=1}m^{-\left(\left|s\right|+\left|t\right|\right)/2-1}=J_{1}\left(s,t\right).
\]

Fix parameters $0<r<1$, $0<c<1$, and $\eta>0$. Define a diagonal
positive definite kernel $E$ by 
\[
E\left(u,v\right):=cr^{\left|u\right|}m^{-\left|u\right|}\delta_{u,v}.
\]
Then $LE=rE$, since for $s\in X$, 
\[
\left(LE\right)\left(s,s\right)=\sum^{m}_{i=1}E\left(is,is\right)=\sum^{m}_{i=1}cr^{\left|is\right|}m^{-\left|is\right|}=rcr^{\left|s\right|}m^{-\left|s\right|}=rE\left(s,s\right).
\]

Now set 
\[
J:=J_{0}+\eta J_{1},\qquad K:=J-E.
\]
Since $J_{0}-E$ is diagonal with nonnegative entries (because $0<c<1$),
and $\eta J_{1}\ge0$, it follows that $K$ is positive definite.
Moreover, $K$ is not diagonal because $J_{1}$ contributes off-diagonal
terms: for $u\ne v$ one has $K\left(u,v\right)=\eta m^{-\left(\left|u\right|+\left|v\right|\right)/2}\ne0$.

Finally, the subinvariance inequality \eqref{eq:2-1} holds: 
\[
LK=LJ-LE=J-rE\ge J-E=K,
\]
and the defect kernel is 
\[
D:=LK-K=E-LE=\left(1-r\right)E\ne0,
\]
which is positive definite. The iterates satisfy 
\[
K_{n}=L^{n}K=L^{n}J-L^{n}E=J-r^{n}E,
\]
so $K_{n}\uparrow J$ pointwise and 
\[
K_{\infty}=J.
\]
Moreover, the defect tower is genuinely nontrivial: 
\[
D_{n}=K_{n+1}-K_{n}=r^{n}\left(1-r\right)E\ne0\qquad\left(n\ge0\right).
\]
The diagonal finiteness condition \eqref{eq:2-2} holds because for
every $s\in X$, 
\[
\sup_{n\ge0}K_{n}\left(s,s\right)\le J\left(s,s\right)=m^{-\left|s\right|}+\eta m^{-\left|s\right|}<\infty.
\]
\end{example}

\begin{rem}
In \prettyref{exa:2-5}, setting $\eta=0$ removes the rank-one term
$J_{1}$ and yields a diagonal example (in which $K$ is diagonal
and $K_{\infty}=J_{0}$). Thus the off-diagonal parameter $\eta>0$
may be viewed as a controlled ``correlation'' perturbation of the
diagonal model which preserves the subinvariance dynamics and the
finiteness condition \eqref{eq:2-2}. 
\end{rem}

\section{Diagonal dynamics}\label{sec:3}

The finiteness hypothesis \eqref{eq:2-2} and the existence of the
invariant completion are governed by the diagonal growth of the tower.
In this section, we isolate the diagonal data $u_{n}\left(s\right)=K_{n}\left(s,s\right)$
and show that it evolves under the induced branching operator $P$
on functions, $u_{n+1}=Pu_{n}$. This reduces diagonal boundedness
to a purely combinatorial problem on the rooted $\phi$-tree. We identify
the envelope $h_{\infty}\left(s\right)=\sup_{n\ge0}u_{n}\left(s\right)$
as the minimal $P$-harmonic majorant of $u_{0}$, introduce the forward-invariant
finiteness region $X_{\mathrm{fin}}$, and derive usable criteria,
via Lyapunov superharmonic functions and branch counting estimates,
for deciding whether $h_{\infty}\left(s\right)$ is finite or infinite.

We begin with
\[
u_{0}\left(s\right):=K\left(s,s\right)\in\left[0,\infty\right],
\]
and define the positive operator $P$ on nonnegative functions $u:X\to\left[0,\infty\right]$
by 
\[
\left(Pu\right)\left(s\right):=\sum^{m}_{i=1}u\left(\phi_{i}\left(s\right)\right).
\]
Recall that, for a word $w=i_{1}\cdots i_{n}\in\mathcal{W}$, we set
\[
\phi_{w}:=\phi_{i_{1}}\circ\cdots\circ\phi_{i_{n}},
\]
and we have $K_{n+1}=LK_{n}$, so $K_{n}=L^{n}K$. Set 
\[
u_{n}\left(s\right):=K_{n}\left(s,s\right).
\]

\begin{lem}
\label{lem:3-1}For every $n\ge0$ and $s\in X$, $u_{n+1}\left(s\right)=\left(Pu_{n}\right)\left(s\right)$,
hence $u_{n}=P^{n}u_{0}$. Equivalently, 
\begin{equation}
u_{n}\left(s\right)=\sum_{\left|w\right|=n}u_{0}\left(\phi_{w}\left(s\right)\right).\label{eq:3-1-1}
\end{equation}
\end{lem}

\begin{proof}
By definition, 
\begin{align*}
u_{n+1}\left(s\right) & =K_{n+1}\left(s,s\right)=\left(LK_{n}\right)\left(s,s\right)\\
 & =\sum^{m}_{i=1}K_{n}\left(\phi_{i}\left(s\right),\phi_{i}\left(s\right)\right)=\sum^{m}_{i=1}u_{n}\left(\phi_{i}\left(s\right)\right)=\left(Pu_{n}\right)\left(s\right).
\end{align*}
Iterating gives $u_{n}=P^{n}u_{0}$, and expanding $P^{n}$ yields
the $\phi$-tree formula \eqref{eq:3-1-1}. 
\end{proof}

A first consequence is that \eqref{eq:2-2} is entirely a $\phi$-driven
growth condition (\prettyref{thm:3-2}), i.e., it is equivalent to
\begin{equation}
\sup_{n\ge0}\left(P^{n}u_{0}\right)\left(s\right)<\infty,\quad\forall s\in X.\label{eq:3-1}
\end{equation}
This already makes the maps $\phi_{i}$ visible: it is not merely
``about $L$'' in the abstract, but about the combinatorics of the
rooted $\phi$-tree $w\mapsto\phi_{w}\left(s\right)$ and how the
diagonal weight $u_{0}\left(s\right)=K\left(s,s\right)$ decays along
it.
\begin{thm}
\label{thm:3-2}Let 
\[
h_{\infty}\left(s\right):=\sup_{n\ge0}u_{n}\left(s\right)=\sup_{n\ge0}\left(P^{n}u_{0}\right)\left(s\right)\in\left[0,\infty\right].
\]
Then the following hold:
\begin{enumerate}
\item $h_{\infty}$ is the minimal $P$-harmonic majorant of $u_{0}$ in
the following sense: if $h:X\to\left[0,\infty\right]$ satisfies 
\[
h\ge u_{0},\qquad Ph=h,
\]
then $h\ge h_{\infty}$. 
\item Let 
\[
X_{\mathrm{fin}}:=\left\{ s\in X:\ h_{\infty}\left(s\right)<\infty\right\} ,\qquad X_{\mathrm{inf}}:=X\setminus X_{\mathrm{fin}}.
\]
Then $X_{\mathrm{fin}}$ is forward invariant for the maps $\phi_{i}$:
if $s\in X_{\mathrm{fin}}$, then $\phi_{i}\left(s\right)\in X_{\mathrm{fin}}$
for every $i$. Equivalently, if $\phi_{i}\left(s\right)\in X_{\mathrm{inf}}$
for some $i$, then $s\in X_{\mathrm{inf}}$. 
\item The finiteness hypothesis \eqref{eq:2-2} holds if and only if $X_{\mathrm{fin}}=X$. 
\end{enumerate}
\end{thm}

\begin{proof}
Fix $h$ with $h\ge u_{0}$ and $Ph=h$. Then we have $P^{n}h=h$
for all $n\ge0$. Moreover, since $P$ is positive, $h\ge u_{0}$
implies $P^{n}h\ge P^{n}u_{0}=u_{n}$ for all $n$. Hence $h\ge u_{n}$
for every $n$, and therefore 
\[
h\ge\sup_{n\ge0}u_{n}=h_{\infty}.
\]
This proves the minimality statement in (1). Note that, since $LK\ge K$,
the kernels $K_{n+1}-K_{n}$ are positive definite, hence $u_{n+1}\left(s\right)-u_{n}\left(s\right)=\left(K_{n+1}-K_{n}\right)\left(s,s\right)\ge0$
for every $s$. Thus $u_{n}\left(s\right)\uparrow h_{\infty}\left(s\right)$
pointwise, and because $m$ is finite we may pass to the limit in
the defining sum for $P$: for every $s\in X$, 
\begin{align*}
\left(Ph_{\infty}\right)\left(s\right) & =\sum^{m}_{i=1}h_{\infty}\left(\phi_{i}\left(s\right)\right)=\sum^{m}_{i=1}\lim_{n\to\infty}u_{n}\left(\phi_{i}\left(s\right)\right)\\
 & =\lim_{n\to\infty}\sum^{m}_{i=1}u_{n}\left(\phi_{i}\left(s\right)\right)=\lim_{n\to\infty}u_{n+1}\left(s\right)=h_{\infty}\left(s\right).
\end{align*}
In particular, $h_{\infty}$ itself is $P$-harmonic. 

For (2), assume $h_{\infty}\left(s\right)<\infty$. Then $u_{n+1}\left(s\right)=\sum_{i}u_{n}\left(\phi_{i}\left(s\right)\right)\le h_{\infty}\left(s\right)$
for all $n$. Since each term is nonnegative, for each fixed $i$
we have $u_{n}\left(\phi_{i}\left(s\right)\right)\le u_{n+1}\left(s\right)\le h_{\infty}\left(s\right)$
for all $n$, hence $h_{\infty}\left(\phi_{i}\left(s\right)\right)<\infty$.
This is the forward invariance.

Finally, (3) is immediate from the definition: \eqref{eq:2-2} is
the pointwise finiteness of $\sup_{n}u_{n}\left(s\right)$, i.e.,
$X_{\mathrm{fin}}=X$. 
\end{proof}

\begin{rem}
The diagonal growth dichotomy governs where the invariant completion
from \prettyref{thm:2-1} is genuinely finite. If $s\in X_{\mathrm{inf}}$,
then $u_{n}\left(s\right)=K_{n}\left(s,s\right)\to\infty$, so the
diagonal of the tower diverges at $s$ and the finiteness hypothesis
\eqref{eq:2-2} fails there. On the other hand, on $X_{\mathrm{fin}}$
one has $\sup_{n}K_{n}\left(s,s\right)<\infty$, and the Gram-matrix
argument in \prettyref{thm:2-1}(1) applies on each finite subset
$F\subset X_{\mathrm{fin}}$, giving a well-defined limit kernel $K_{\infty}$
on $X_{\mathrm{fin}}\times X_{\mathrm{fin}}$ with $K_{\infty}\ge K$
and $LK_{\infty}=K_{\infty}$ restricted to $X_{\mathrm{fin}}$. Thus,
\prettyref{sec:2} may be viewed as operating on the forward-invariant
region $X_{\mathrm{fin}}$, while the present section provides dynamical
criteria to determine when $X_{\mathrm{fin}}=X$. 
\end{rem}

A useful conceptual point is that this theorem is already a sharp
``potential-theoretic'' characterization driven by the $\phi$-tree:
the diagonal growth is finite exactly when the diagonal data $u_{0}\left(s\right)=K\left(s,s\right)$
lies under a harmonic function for the branching operator $P$.

The abstract characterization above becomes usable once one has concrete
sufficient conditions phrased directly in terms of the maps $\phi_{i}$.
The following criterion (\prettyref{cor:3-4}) makes the $\phi_{i}$
explicit: it asks for a ``Lyapunov function'' $r$ that does not
grow (on average) along the $\phi$-branches.
\begin{cor}
\label{cor:3-4}Assume there exists a function $r:X\to\left(0,\infty\right)$
and a constant $C<\infty$ such that 
\[
u_{0}\le Cr,\qquad Pr\le r.
\]
Then $h_{\infty}\left(s\right)<\infty$ for all $s$, hence \eqref{eq:3-1}
holds. 
\end{cor}

\begin{proof}
We show by induction that $u_{n}\le Cr$ for all $n$. For $n=0$
this is $u_{0}\le Cr$. If $u_{n}\le Cr$, then 
\[
u_{n+1}=Pu_{n}\le CPr\le Cr.
\]
Thus $u_{n}\left(s\right)\le Cr\left(s\right)$ for all $n$, so $h_{\infty}\left(s\right)\le Cr\left(s\right)<\infty$. 
\end{proof}

The next result (\prettyref{cor:3-5}) is a complementary ``non-decay''
mechanism: if there is a region where the diagonal kernel mass $u_{0}\left(s\right)=K\left(s,s\right)$
is bounded below, and the $\phi$-tree hits that region along exponentially
many branches, then the diagonal growth must blow up.
\begin{cor}
\label{cor:3-5}Fix $s\in X$. Suppose there exists a subset $Y\subset X$,
constants $\varepsilon>0$, $\rho>1$, and an infinite sequence $n_{k}\to\infty$
such that 
\[
u_{0}\left(y\right)\ge\varepsilon\ \text{for all }y\in Y,\qquad\#\left\{ w:\ \left|w\right|=n_{k},\ \phi_{w}\left(s\right)\in Y\right\} \ \ge\ \rho^{n_{k}}\ \text{for all }k.
\]
Then $h_{\infty}\left(s\right)=\infty$, so $s\in X_{\mathrm{inf}}$. 
\end{cor}

\begin{proof}
By the $\phi$-tree expansion, 
\begin{align*}
u_{n_{k}}\left(s\right) & =\sum_{\left|w\right|=n_{k}}u_{0}\left(\phi_{w}\left(s\right)\right)\\
 & \ge\sum_{\substack{\left|w\right|=n_{k}\\
\phi_{w}\left(s\right)\in Y
}
}u_{0}\left(\phi_{w}\left(s\right)\right)\ge\varepsilon\ \#\left\{ w:\left|w\right|=n_{k},\ \phi_{w}\left(s\right)\in Y\right\} \ge\varepsilon\rho^{n_{k}}.
\end{align*}
Hence $u_{n_{k}}\left(s\right)\to\infty$, so $h_{\infty}\left(s\right)=\infty$. 
\end{proof}

We now characterize the diagonal growth entirely in terms of $\phi$-tree
level-set counts. 
\begin{thm}
\label{thm:3-6} Fix $s\in X$. For each $\theta>0$ and $n\ge0$
define 
\[
N_{n}\left(s,\theta\right):=\#\left\{ w:\left|w\right|=n,u_{0}\left(\phi_{w}\left(s\right)\right)\ge\theta\right\} ,
\]
where $u_{0}\left(s\right)=K\left(s,s\right)$ and $\phi_{w}=\phi_{i_{1}}\circ\cdots\circ\phi_{i_{n}}$
for a word $w=i_{1}\cdots i_{n}\in\mathcal{W}$. Then: 
\begin{enumerate}
\item For every $n\ge0$, 
\begin{equation}
u_{n}\left(s\right)=\int^{\infty}_{0}N_{n}\left(s,\theta\right)d\theta.\label{eq:3-2}
\end{equation}
\item Consequently, 
\begin{equation}
h_{\infty}\left(s\right)=\sup_{n\ge0}u_{n}\left(s\right)\le\int^{\infty}_{0}\left(\sup_{n\ge0}N_{n}\left(s,\theta\right)\right)d\theta.\label{eq:3-3}
\end{equation}
\item In particular: 
\begin{enumerate}
\item If there exists $\theta_{0}>0$ such that $\sup_{n\ge0}N_{n}\left(s,\theta_{0}\right)=\infty$,
then $h_{\infty}\left(s\right)=\infty$, hence $s\in X_{\mathrm{inf}}$. 
\item If 
\begin{equation}
\int^{\infty}_{0}\left(\sup_{n\ge0}N_{n}\left(s,\theta\right)\right)d\theta<\infty,\label{eq:3-4}
\end{equation}
then $h_{\infty}\left(s\right)<\infty$, hence $s\in X_{\mathrm{fin}}$
and so \eqref{eq:3-1} holds at $s$. 
\end{enumerate}
\end{enumerate}
\end{thm}

\begin{proof}
Fix $s\in X$ and $n\ge0$. By \prettyref{lem:3-1},
\[
u_{n}\left(s\right)=\sum_{\left|w\right|=n}u_{0}\left(\phi_{w}\left(s\right)\right).
\]
For each $a\in\left[0,\infty\right)$ one has the identity 
\[
a=\int^{\infty}_{0}\mathbf{1}_{\left\{ a\ge\theta\right\} }d\theta.
\]
Applying this with $a=u_{0}\left(\phi_{w}\left(s\right)\right)$ and
summing over $\left|w\right|=n$ gives 
\begin{align*}
u_{n}\left(s\right) & =\sum_{\left|w\right|=n}\int^{\infty}_{0}\mathbf{1}_{\left\{ u_{0}\left(\phi_{w}\left(s\right)\right)\ge\theta\right\} }d\theta\\
 & =\int^{\infty}_{0}\left(\sum_{\left|w\right|=n}\mathbf{1}_{\left\{ u_{0}\left(\phi_{w}\left(s\right)\right)\ge\theta\right\} }\right)d\theta=\int^{\infty}_{0}N_{n}\left(s,\theta\right)d\theta,
\end{align*}
where we interchange sum and integral by nonnegativity. This proves
\eqref{eq:3-2}.

For \eqref{eq:3-3}, observe that for each $n$, 
\[
\int^{\infty}_{0}N_{n}\left(s,\theta\right)d\theta\le\int^{\infty}_{0}\left(\sup_{k\ge0}N_{k}\left(s,\theta\right)\right)d\theta.
\]
Taking the supremum over $n$ yields \eqref{eq:3-3}.

For (3), if $\sup_{n}N_{n}\left(s,\theta_{0}\right)=\infty$, then
for every $n$, using that $\theta\mapsto N_{n}\left(s,\theta\right)$
is non-increasing, 
\begin{align*}
u_{n}\left(s\right) & =\int^{\infty}_{0}N_{n}\left(s,\theta\right)d\theta\\
 & \ge\int^{\theta_{0}}_{0}N_{n}\left(s,\theta\right)d\theta\ge\int^{\theta_{0}}_{0}N_{n}\left(s,\theta_{0}\right)d\theta=\theta_{0}N_{n}\left(s,\theta_{0}\right),
\end{align*}
so $\sup_{n}u_{n}\left(s\right)=\infty$, i.e. $h_{\infty}\left(s\right)=\infty$.
The finiteness condition in \eqref{eq:3-4} implies $h_{\infty}\left(s\right)<\infty$
by \eqref{eq:3-3}, hence $s\in X_{\mathrm{fin}}$. 
\end{proof}

\begin{rem}
Corollaries \ref{cor:3-4} and \ref{cor:3-5} are obtained by estimating
$N_{n}\left(s,\theta\right)$ from below on a fixed level set $\left\{ u_{0}\ge\theta\right\} $
(the non-decay mechanism) and from above by dominating $u_{0}$ by
a $P$-superharmonic Lyapunov function (the decay mechanism), which
in turn forces uniform control of the integrand in \eqref{eq:3-3}.
\end{rem}

The next estimate explains why the diagonal criteria from this section
are already useful for controlling the \emph{full} invariant completion:
once the diagonal tail $h_{\infty}-u_{N}$ is small, the entire kernel
tail $K_{\infty}-K_{N}$ is small uniformly on finite sets.
\begin{thm}
\label{thm:3-8} Fix $N\ge0$. On $X_{\mathrm{fin}}$ the limit kernel
$K_{\infty}$ from \prettyref{thm:2-1} is well-defined, and the tail
kernel 
\[
R_{N}\left(s,t\right):=K_{\infty}\left(s,t\right)-K_{N}\left(s,t\right),\qquad s,t\in X_{\mathrm{fin}},
\]
is positive definite. Moreover, for all $s,t\in X_{\mathrm{fin}}$
one has 
\begin{equation}
\left|K_{\infty}\left(s,t\right)-K_{N}\left(s,t\right)\right|^{2}\le\left(K_{\infty}\left(s,s\right)-K_{N}\left(s,s\right)\right)\left(K_{\infty}\left(t,t\right)-K_{N}\left(t,t\right)\right).\label{eq:3-5}
\end{equation}
Equivalently, writing $u_{N}\left(s\right)=K_{N}\left(s,s\right)$
and $h_{\infty}\left(s\right)=K_{\infty}\left(s,s\right)$ on $X_{\mathrm{fin}}$,
\begin{equation}
\left|K_{\infty}\left(s,t\right)-K_{N}\left(s,t\right)\right|\le\sqrt{\left(h_{\infty}\left(s\right)-u_{N}\left(s\right)\right)\left(h_{\infty}\left(t\right)-u_{N}\left(t\right)\right)}.\label{eq:3-6}
\end{equation}
\end{thm}

\begin{proof}
Fix a finite subset $F=\left\{ x_{1},\dots,x_{r}\right\} \subset X_{\mathrm{fin}}$.
Since $x_{a}\in X_{\mathrm{fin}}$, the diagonal sequence $K_{n}\left(x_{a},x_{a}\right)$
is bounded in $n$, hence the Gram-matrix argument from \prettyref{thm:2-1}(1)
applies on $F$ and yields the limit matrix 
\[
G_{\infty}:=\lim_{n\to\infty}\left[K_{n}\left(x_{a},x_{b}\right)\right]_{a,b}.
\]
In particular, $K_{\infty}$ exists on $F\times F$ and the Gram matrices
satisfy 
\[
\left[K_{N}\left(x_{a},x_{b}\right)\right]_{a,b}\le\left[K_{\infty}\left(x_{a},x_{b}\right)\right]_{a,b}
\]
in the order on Hermitian matrices (because $\left\{ K_{n}\right\} $
is increasing). Therefore, for each $N$, the difference matrix 
\[
\left[R_{N}\left(x_{a},x_{b}\right)\right]_{a,b}=\left[K_{\infty}\left(x_{a},x_{b}\right)-K_{N}\left(x_{a},x_{b}\right)\right]_{a,b}
\]
is positive semidefinite. Since this holds for every finite $F\subset X_{\mathrm{fin}}$,
it follows that $R_{N}\in\mathcal{K}\left(X_{\mathrm{fin}}\right)$.

Let $\mathcal{H}_{R_{N}}$ be the RKHS of $R_{N}$. For $s,t\in X_{\mathrm{fin}}$,
the reproducing property gives 
\[
R_{N}\left(s,t\right)=\left\langle \left(R_{N}\right)_{s},\left(R_{N}\right)_{t}\right\rangle _{\mathcal{H}_{R_{N}}}.
\]
Hence, by Cauchy-Schwarz, 
\[
\left|R_{N}\left(s,t\right)\right|^{2}\le\left\Vert \left(R_{N}\right)_{s}\right\Vert ^{2}_{\mathcal{H}_{R_{N}}}\left\Vert \left(R_{N}\right)_{t}\right\Vert ^{2}_{\mathcal{H}_{R_{N}}}=R_{N}\left(s,s\right)R_{N}\left(t,t\right),
\]
using $\left\Vert J_{s}\right\Vert ^{2}_{\mathcal{H}_{J}}=J\left(s,s\right)$
for any p.d. kernel $J$. Substituting $R_{N}=K_{\infty}-K_{N}$ yields
\eqref{eq:3-5}, and \eqref{eq:3-6} is just the same inequality with
$R_{N}\left(s,s\right)=h_{\infty}\left(s\right)-u_{N}\left(s\right)$. 
\end{proof}

A convenient way to use \prettyref{thm:3-8} is to bound the diagonal
tail $h_{\infty}-u_{N}$ via the defect tower.
\begin{cor}
\label{cor:3-9} Let $d_{n}\left(s\right):=D_{n}\left(s,s\right)$
and $d_{0}\left(s\right)=D\left(s,s\right)$. Then on $X_{\mathrm{fin}}$,
for every $N\ge0$ one has 
\begin{equation}
h_{\infty}\left(s\right)-u_{N}\left(s\right)=\sum_{n\ge N}d_{n}\left(s\right).\label{eq:3-7}
\end{equation}
In particular, suppose there exist a function $r:X\to\left(0,\infty\right)$,
constants $C<\infty$ and $\beta\in\left(0,1\right)$ such that 
\[
d_{0}\le Cr,\qquad Pr\le\beta r.
\]
Then for every $N\ge0$ and $s,t\in X_{\mathrm{fin}}$, 
\begin{equation}
\left|K_{\infty}\left(s,t\right)-K_{N}\left(s,t\right)\right|\le\frac{C}{1-\beta}\beta^{N}\sqrt{r\left(s\right)r\left(t\right)}.\label{eq:3-8}
\end{equation}
\end{cor}

\begin{proof}
Since $K_{\infty}=K+\sum_{n\ge0}D_{n}$ in Gram-matrix sense on $X_{\mathrm{fin}}$
(see \prettyref{thm:2-1}), taking diagonals gives 
\[
h_{\infty}\left(s\right)-u_{N}\left(s\right)=K_{\infty}\left(s,s\right)-K_{N}\left(s,s\right)=\sum_{n\ge N}D_{n}\left(s,s\right)=\sum_{n\ge N}d_{n}\left(s\right),
\]
which is \eqref{eq:3-7}. Next, note that $d_{n}\left(s\right)=D_{n}\left(s,s\right)=\left(L^{n}D\right)\left(s,s\right)$,
so by the same diagonal computation as in \prettyref{lem:3-1} one
has 
\[
d_{n+1}=Pd_{n},\qquad d_{n}=P^{n}d_{0}.
\]
Hence $d_{0}\le Cr$ and $Pr\le\beta r$ imply inductively that 
\[
d_{n}=P^{n}d_{0}\le CP^{n}r\le C\beta^{n}r.
\]
Therefore 
\[
h_{\infty}\left(s\right)-u_{N}\left(s\right)=\sum_{n\ge N}d_{n}\left(s\right)\le\sum_{n\ge N}C\beta^{n}r\left(s\right)=\frac{C}{1-\beta}\beta^{N}r\left(s\right).
\]
Combining this diagonal tail estimate with \prettyref{thm:3-8} gives
\eqref{eq:3-8}. 
\end{proof}

\section{Gaussian defect martingales}\label{sec:4}

\prettyref{thm:2-1} gives a canonical splitting of the invariant
completion $K_{\infty}$ into the initial kernel $K$ and the defect
kernels $D_{n}$. In this section we turn that splitting into a probabilistic
object: a Gaussian process whose level truncations form a martingale,
and whose predictable quadratic variation is the defect tower. This
gives a concrete Radon-Nikodym interpretation of the compression operator
$A$.
\begin{thm}
\label{thm:4-1} Assume the setting of \prettyref{sec:2}, and define
\[
\mathcal{E}:=\mathcal{H}_{K}\oplus\bigoplus_{n\ge0}\mathcal{H}_{D_{n}},\qquad v\left(s\right):=K_{s}\oplus\bigoplus_{n\ge0}\left(D_{n}\right)_{s}\in\mathcal{E},\qquad A:=P^{*}_{0}P_{0}\in B\left(\mathcal{E}\right)
\]
as in \prettyref{thm:2-1}(3). Let $\left(G_{0},G_{1},G_{2},\dots\right)$
be a sequence of independent centered complex Gaussian Hilbert-space
vectors with 
\[
G_{0}\in\mathcal{H}_{K},\qquad G_{n+1}\in\mathcal{H}_{D_{n}}\quad\left(n\ge0\right),
\]
each normalized so that for all $f,g$ in the corresponding Hilbert
space, 
\[
\mathbb{E}\left[\overline{\left\langle f,G\right\rangle }\left\langle g,G\right\rangle \right]=\left\langle f,g\right\rangle .
\]
For $N\ge0$ define the truncated random field 
\[
X_{N}\left(s\right):=\left\langle K_{s},G_{0}\right\rangle _{\mathcal{H}_{K}}+\sum^{N-1}_{n=0}\left\langle \left(D_{n}\right)_{s},G_{n+1}\right\rangle _{\mathcal{H}_{D_{n}}},\qquad s\in X,
\]
with the convention that the sum is $0$ when $N=0$. Let $\mathcal{F}_{N}$
be the $\sigma$-algebra generated by $\left(G_{0},G_{1},\dots,G_{N}\right)$.
Then the following hold.
\begin{enumerate}
\item For every $N\ge0$ and all $s,t\in X$, 
\[
\mathbb{E}\left[\overline{X_{N}\left(s\right)}X_{N}\left(t\right)\right]=K_{N}\left(s,t\right).
\]
In particular, each finite-dimensional distribution of $\left(X_{N}\left(s\right)\right)_{s\in X}$
is centered complex Gaussian with covariance kernel $K_{N}$.
\item For each fixed $s\in X$, the sequence $\left(X_{N}\left(s\right)\right)_{N\ge0}$
is an $\left(\mathcal{F}_{N}\right)$-martingale in $L^{2}$, with
orthogonal increments 
\[
X_{N+1}\left(s\right)-X_{N}\left(s\right)=\left\langle \left(D_{N}\right)_{s},G_{N+1}\right\rangle _{\mathcal{H}_{D_{N}}}.
\]
Moreover, for all $s,t\in X$, 
\[
\mathbb{E}\left[\overline{\left(X_{N+1}\left(s\right)-X_{N}\left(s\right)\right)}\left(X_{N+1}\left(t\right)-X_{N}\left(t\right)\right)\ \middle|\ \mathcal{F}_{N}\right]=D_{N}\left(s,t\right),
\]
thus the predictable quadratic variation of the martingale is the
defect tower.
\item Fix a finite subset $F=\left\{ x_{1},\dots,x_{r}\right\} \subset X$.
Then $\left(X_{N}\left(x_{a}\right)\right)^{r}_{a=1}$ converges in
$L^{2}$ (and almost surely) as $N\to\infty$ if and only if the Gram
matrices $\left[K_{N}\left(x_{a},x_{b}\right)\right]_{a,b}$ are uniformly
bounded, equivalently if and only if the limit Gram matrix 
\[
G_{\infty}:=\lim_{N\to\infty}\left[K_{N}\left(x_{a},x_{b}\right)\right]_{a,b}
\]
exists as in the Gram-matrix argument of \prettyref{thm:2-1}(1).
In that case the limit 
\[
X_{\infty}\left(x_{a}\right):=\lim_{N\to\infty}X_{N}\left(x_{a}\right)
\]
is centered complex Gaussian and satisfies 
\[
\mathbb{E}\left[\overline{X_{\infty}\left(x_{a}\right)}X_{\infty}\left(x_{b}\right)\right]=K_{\infty}\left(x_{a},x_{b}\right)\qquad\left(1\le a,b\le r\right).
\]
\item Let $G:=G_{0}\oplus\bigoplus_{n\ge0}G_{n+1}\in\mathcal{E}$, and define
two random fields on $X$ by 
\[
Z\left(s\right):=\left\langle v\left(s\right),G\right\rangle _{\mathcal{E}},\qquad Y\left(s\right):=\left\langle v\left(s\right),A^{1/2}G\right\rangle _{\mathcal{E}}.
\]
Then $Z$ is centered complex Gaussian with covariance $K_{\infty}$,
while $Y$ is centered complex Gaussian with covariance $K$, namely
\[
\mathbb{E}\left[\overline{Z\left(s\right)}Z\left(t\right)\right]=K_{\infty}\left(s,t\right),\qquad\mathbb{E}\left[\overline{Y\left(s\right)}Y\left(t\right)\right]=K\left(s,t\right).
\]
Thus, $A$ is the Radon-Nikodym compression operator in the sense
of a Gaussian covariance contraction. 
\end{enumerate}
\end{thm}

\begin{proof}
(1) Fix $N$ and $s,t\in X$. Independence of the Gaussian vectors
across levels gives 
\[
\mathbb{E}\left[\overline{X_{N}\left(s\right)}X_{N}\left(t\right)\right]=\mathbb{E}\left[\overline{\left\langle K_{s},G_{0}\right\rangle }\left\langle K_{t},G_{0}\right\rangle \right]+\sum^{N-1}_{n=0}\mathbb{E}\left[\overline{\left\langle \left(D_{n}\right)_{s},G_{n+1}\right\rangle }\left\langle \left(D_{n}\right)_{t},G_{n+1}\right\rangle \right].
\]
By the normalization of each Gaussian vector and the RKHS identity
$\left\langle J_{s},J_{t}\right\rangle _{\mathcal{H}_{J}}=J\left(s,t\right)$,
the first term equals $K\left(s,t\right)$ and the $n$th term in
the sum equals $D_{n}\left(s,t\right)$. Hence 
\[
\mathbb{E}\left[\overline{X_{N}\left(s\right)}X_{N}\left(t\right)\right]=K\left(s,t\right)+\sum^{N-1}_{n=0}D_{n}\left(s,t\right)=K_{N}\left(s,t\right)
\]
by \eqref{eq:2-3}. This proves (1).

(2) Fix $s\in X$. Since $X_{N}\left(s\right)$ depends only on $\left(G_{0},\dots,G_{N}\right)$,
it is $\mathcal{F}_{N}$-measurable. Moreover, 
\[
X_{N+1}\left(s\right)=X_{N}\left(s\right)+\left\langle \left(D_{N}\right)_{s},G_{N+1}\right\rangle .
\]
Because $G_{N+1}$ is independent of $\mathcal{F}_{N}$ and centered,
\[
\mathbb{E}\left[X_{N+1}\left(s\right)\ \middle|\ \mathcal{F}_{N}\right]=X_{N}\left(s\right)+\mathbb{E}\left[\left\langle \left(D_{N}\right)_{s},G_{N+1}\right\rangle \right]=X_{N}\left(s\right),
\]
so $\left(X_{N}\left(s\right)\right)$ is a martingale. Orthogonality
of increments follows from independence across levels. For the conditional
covariance identity, again using independence and the normalization,
\[
\mathbb{E}\left[\overline{\left\langle \left(D_{N}\right)_{s},G_{N+1}\right\rangle }\left\langle \left(D_{N}\right)_{t},G_{N+1}\right\rangle \ \middle|\ \mathcal{F}_{N}\right]=\left\langle \left(D_{N}\right)_{s},\left(D_{N}\right)_{t}\right\rangle _{\mathcal{H}_{D_{N}}}=D_{N}\left(s,t\right),
\]
which is the claimed predictable quadratic variation identity.

(3) Fix $F=\left\{ x_{1},\dots,x_{r}\right\} $. By (1), the covariance
matrix of the centered Gaussian vector $\left(X_{N}\left(x_{a}\right)\right)^{r}_{a=1}$
is the Gram matrix $G_{N}=\left[K_{N}\left(x_{a},x_{b}\right)\right]_{a,b}$.
The sequence converges in $L^{2}$ (equivalently, is Cauchy in $L^{2}$)
if and only if $\sup_{N}\left\Vert G_{N}\right\Vert <\infty$, which
is equivalent to the existence of the limit Gram matrix $G_{\infty}$
as in the Gram-matrix argument of \prettyref{thm:2-1}(1). In that
case the limit is necessarily centered Gaussian with covariance $G_{\infty}$,
hence corresponds to $K_{\infty}$ on $F\times F$.

(4) The random vector $G\in\mathcal{E}$ is centered Gaussian with
covariance equal to the identity on $\mathcal{E}$ in the same sense
as above. Therefore, for $Z\left(s\right)=\left\langle v\left(s\right),G\right\rangle $,
\[
\mathbb{E}\left[\overline{Z\left(s\right)}Z\left(t\right)\right]=\left\langle v\left(s\right),v\left(t\right)\right\rangle _{\mathcal{E}}=K_{\infty}\left(s,t\right)
\]
by \prettyref{thm:2-1}(3). Similarly, since $A^{1/2}$ is a contraction,
\[
\mathbb{E}\left[\overline{Y\left(s\right)}Y\left(t\right)\right]=\left\langle v\left(s\right),Av\left(t\right)\right\rangle _{\mathcal{E}}=K\left(s,t\right),
\]
again by \prettyref{thm:2-1}(3). This completes the proof. 
\end{proof}

\begin{rem}
\label{rem:4-2} The theorem is a probabilistic interpretation of
the defect tower: $X_{N}$ is the level-$N$ partial energy, and $D_{N}$
is  the conditional covariance of the $N$th increment. In particular,
on each finite subset $F\subset X$ the convergence of the invariant
completion $K_{\infty}$ is equivalent to the $L^{2}$-boundedness
of this martingale, so the finiteness hypothesis \eqref{eq:2-2} is
a martingale boundedness condition expressed on diagonals. 
\end{rem}

\section{Boundary model and Doob transform}\label{sec:5}

We now make explicit the tree/pathspace picture which has been implicit
in our discussion, and we obtain the canonical path measures attached
to the minimal harmonic majorant $h_{\infty}$ from \prettyref{sec:3}.
The point is that the diagonal harmonic object $h_{\infty}$ produces
a natural family of Markov measures on the boundary $\Omega$, and
these measures organize the iterates of $L$ (after a canonical normalization)
as boundary averages.

Throughout this section we assume the standing hypotheses of Sections
\ref{sec:2}--\ref{sec:4}, and we write 
\[
h:=h_{\infty}
\]
for the minimal $P$-harmonic majorant of $u_{0}\left(s\right)=K\left(s,s\right)$
(constructed in \prettyref{sec:3}). We recall that $h$ is $P$-harmonic,
meaning 
\[
h\left(s\right)=\sum^{m}_{i=1}h\left(\varphi_{i}\left(s\right)\right),
\]
and that $h\left(s\right)\in\left(0,\infty\right)$ precisely on the
set $X_{\mathrm{fin}}$ where the diagonal growth is finite. In what
follows we work on 
\[
X^{+}_{\mathrm{fin}}:=\left\{ s\in X_{\mathrm{fin}}:h\left(s\right)>0\right\} ,
\]
where $K_{\infty}$ is defined by the finite-set Gram-matrix limits.

Let 
\[
\Omega:=\left\{ 1,\dots,m\right\} ^{\mathbb{N}}
\]
be the pathspace, equipped with its product $\sigma$-algebra $\mathcal{F}$
generated by cylinder sets. For a finite word $w=i_{1}\cdots i_{n}\in\left\{ 1,\dots,m\right\} ^{n}$
we write 
\[
\left[w\right]:=\left\{ \omega\in\Omega:\omega_{1}\cdots\omega_{n}=w\right\} .
\]
We also write $\omega|n=\omega_{1}\cdots\omega_{n}$ for the length-$n$
prefix. For $s\in X$ and $w=i_{1}\cdots i_{n}$ we keep the usual
composition notation 
\[
\varphi_{w}=\varphi_{i_{1}}\circ\cdots\circ\varphi_{i_{n}},\qquad\varphi_{\emptyset}=\mathrm{id}_{X},
\]
and we also define the reversed composition 
\[
\check{\varphi}_{w}:=\varphi_{i_{n}}\circ\cdots\circ\varphi_{i_{1}},\qquad\check{\varphi}_{\emptyset}=\mathrm{id}_{X}.
\]
We then define the branch points along $\omega$ by 
\[
s_{n}\left(\omega\right):=\check{\varphi}_{\omega|n}\left(s\right),\qquad n\geq0.
\]

The harmonic function $h$ defines canonical transition probabilities
on $X_{\mathrm{fin}}$. For $s\in X^{+}_{\mathrm{fin}}$, set 
\[
p_{i}\left(s\right):=\frac{h\left(\varphi_{i}\left(s\right)\right)}{h\left(s\right)},\qquad i=1,\dots,m.
\]
Then $p_{i}\left(s\right)\geq0$ and $\sum^{m}_{i=1}p_{i}\left(s\right)=1$
for every $s\in X^{+}_{\mathrm{fin}}$, by $P$-harmonicity of $h$.
\begin{prop}
\label{prop:5-1} For each $s\in X^{+}_{\mathrm{fin}}$ there is a
unique probability measure $\mu_{s}$ on $\left(\Omega,\mathcal{F}\right)$
such that for every word $w=i_{1}\cdots i_{n}$, 
\begin{equation}
\mu_{s}\left(\left[w\right]\right)=\frac{h\left(\check{\varphi}_{w}\left(s\right)\right)}{h\left(s\right)}=\prod^{n}_{k=1}p_{i_{k}}\left(\check{\varphi}_{i_{1}\cdots i_{k-1}}\left(s\right)\right).\label{eq:5-1}
\end{equation}
Moreover, if $\mathcal{F}_{n}$ denotes the $\sigma$-algebra generated
by cylinders of length $n$, then under $\mu_{s}$ the process $\left(s_{n}\left(\omega\right)\right)_{n\geq0}$
is a time-homogeneous Markov chain with transitions 
\begin{equation}
\mu_{s}\left(s_{n+1}=\varphi_{i}\left(x\right)\mid s_{n}=x\right)=p_{i}\left(x\right),\qquad x\in X^{+}_{\mathrm{fin}}.\label{eq:5-2}
\end{equation}
Equivalently, for every bounded function $f$ on $X^{+}_{\mathrm{fin}}$
and every $n\geq0$, 
\begin{equation}
\mathbb{E}_{\mu_{s}}\left[f\left(s_{n+1}\right)\mid\mathcal{F}_{n}\right]=\sum^{m}_{i=1}p_{i}\left(s_{n}\right)f\left(\varphi_{i}\left(s_{n}\right)\right).\label{eq:5-3}
\end{equation}
\end{prop}

\begin{proof}
Define a set function on cylinders by 
\[
\mu_{s}\left(\left[w\right]\right):=\frac{h\left(\check{\varphi}_{w}\left(s\right)\right)}{h\left(s\right)}.
\]
Consistency follows from harmonicity: 
\begin{align*}
\sum^{m}_{i=1}\mu_{s}\left(\left[wi\right]\right) & =\sum^{m}_{i=1}\frac{h\left(\check{\varphi}_{wi}\left(s\right)\right)}{h\left(s\right)}\\
 & =\frac{\sum^{m}_{i=1}h\left(\varphi_{i}\left(\check{\varphi}_{w}\left(s\right)\right)\right)}{h\left(s\right)}=\frac{h\left(\check{\varphi}_{w}\left(s\right)\right)}{h\left(s\right)}=\mu_{s}\left(\left[w\right]\right).
\end{align*}
Hence $\mu_{s}$ defines a unique probability measure on $\left(\Omega,\mathcal{F}\right)$
by Kolmogorov extension. The product formula \eqref{eq:5-1} follows:
\[
\frac{h\left(\check{\varphi}_{w}\left(s\right)\right)}{h\left(s\right)}=\prod^{n}_{k=1}\frac{h\left(\check{\varphi}_{i_{1}\cdots i_{k}}\left(s\right)\right)}{h\left(\check{\varphi}_{i_{1}\cdots i_{k-1}}\left(s\right)\right)}=\prod^{n}_{k=1}p_{i_{k}}\left(\check{\varphi}_{i_{1}\cdots i_{k-1}}\left(s\right)\right).
\]
The Markov property is immediate from the cylinder specification.
Indeed, conditioning on $\mathcal{F}_{n}$ is conditioning on the
prefix $\omega|n$, hence on the current state $s_{n}=\check{\varphi}_{\omega|n}\left(s\right)$,
and the conditional distribution of $\omega_{n+1}$ is given by $p_{i}\left(s_{n}\right)$,
which is \eqref{eq:5-2} . The identity \eqref{eq:5-3} is the standard
reformulation of this transition rule. 
\end{proof}

It is convenient to isolate the Doob transformed averaging operator
on functions induced by the measures $\mu_{s}$. Define, for bounded
$f$ on $X^{+}_{\mathrm{fin}}$, 
\begin{equation}
\left(Qf\right)\left(s\right):=\sum^{m}_{i=1}p_{i}\left(s\right)f\left(\varphi_{i}\left(s\right)\right)=\frac{1}{h\left(s\right)}\sum^{m}_{i=1}h\left(\varphi_{i}\left(s\right)\right)f\left(\varphi_{i}\left(s\right)\right).\label{eq:5-4}
\end{equation}
Then $Q$ is a Markov operator on functions on $X^{+}_{\mathrm{fin}}$,
and it is intertwined with the original branching operator $P$ by
the gauge $h$.
\begin{lem}
\label{lem:5-2} For every bounded $f$ on $X^{+}_{\mathrm{fin}}$,
\[
P\left(hf\right)=h\left(Qf\right).
\]
Consequently, for every $n\geq0$, 
\[
P^{n}\left(hf\right)=h\left(Q^{n}f\right),\qquad Q^{n}f\left(s\right)=\mathbb{E}_{\mu_{s}}\left[f\left(s_{n}\right)\right].
\]
\end{lem}

\begin{proof}
By definition (using \eqref{eq:5-4}), 
\[
P\left(hf\right)\left(s\right)=\sum^{m}_{i=1}h\left(\varphi_{i}\left(s\right)\right)f\left(\varphi_{i}\left(s\right)\right)=h\left(s\right)\left(Qf\right)\left(s\right).
\]
Iterating gives $P^{n}\left(hf\right)=h\left(Q^{n}f\right)$. The
formula $Q^{n}f\left(s\right)=\mathbb{E}_{\mu_{s}}\left[f\left(s_{n}\right)\right]$
is the standard Markov chain identity from \prettyref{prop:5-1}. 
\end{proof}

We now turn to kernels. The next object is the natural normalization
of kernels by $h$, and the corresponding transform of $L$. Given
a kernel $J$ on $X^{+}_{\mathrm{fin}}\times X^{+}_{\mathrm{fin}}$
we define its $h$-normalization by 
\[
J^{\left(h\right)}\left(s,t\right):=\frac{J\left(s,t\right)}{h\left(s\right)h\left(t\right)}.
\]
We then define an operator $\widetilde{L}$ on kernels by 
\begin{equation}
(\widetilde{L}G)\left(s,t\right):=\sum^{m}_{i=1}p_{i}\left(s\right)p_{i}\left(t\right)G\left(\varphi_{i}\left(s\right),\varphi_{i}\left(t\right)\right).\label{eq:5-5}
\end{equation}
This operator is the kernel-level analogue of $Q$, except that the
same symbol $i$ is used in both variables, reflecting the fact that
$L$ branches synchronously in $s$ and $t$.
\begin{prop}
\label{prop:5-3} For every kernel $J$ on $X^{+}_{\mathrm{fin}}\times X^{+}_{\mathrm{fin}}$,
\[
\left(LJ\right)^{\left(h\right)}=\widetilde{L}(J^{\left(h\right)}).
\]
In particular, for the kernel tower $K_{n}=L^{n}K$, 
\[
K^{\left(h\right)}_{n}=\widetilde{L}^{n}(K^{\left(h\right)}),\qquad n\geq0,
\]
and if $K_{\infty}$ is $L$-invariant then $K^{\left(h\right)}_{\infty}$
is $\widetilde{L}$-invariant. 
\end{prop}

\begin{proof}
Compute directly: 
\begin{align*}
\left(LJ\right)^{\left(h\right)}\left(s,t\right) & =\frac{\sum^{m}_{i=1}J\left(\varphi_{i}\left(s\right),\varphi_{i}\left(t\right)\right)}{h\left(s\right)h\left(t\right)}\\
 & =\sum^{m}_{i=1}p_{i}\left(s\right)p_{i}\left(t\right)\frac{J\left(\varphi_{i}\left(s\right),\varphi_{i}\left(t\right)\right)}{h\left(\varphi_{i}\left(s\right)\right)h\left(\varphi_{i}\left(t\right)\right)}.
\end{align*}
The last expression is $(\widetilde{L}(J^{\left(h\right)}))\left(s,t\right)$;
see \eqref{eq:5-5}. Iteration gives $K^{\left(h\right)}_{n}=\widetilde{L}^{n}\left(K^{\left(h\right)}\right)$,
and invariance passes through the same intertwining. 
\end{proof}

The operator $\widetilde{L}$ admits an explicit boundary expansion
in terms of the Doob measures $\mu_{s}$. For a word $w=i_{1}\cdots i_{n}$
we set 
\[
p_{w}\left(s\right):=\mu_{s}\left(\left[w\right]\right)=\frac{h\left(\check{\varphi}_{w}\left(s\right)\right)}{h\left(s\right)}.
\]
Then $p_{w}\left(s\right)\geq0$ and $\sum_{|w|=n}p_{w}\left(s\right)=1$
for each fixed $n$.
\begin{lem}
\label{lem:5-4} For every kernel $G$ on $X^{+}_{\mathrm{fin}}\times X^{+}_{\mathrm{fin}}$
and every $n\geq0$, 
\[
(\widetilde{L}^{n}G)\left(s,t\right)=\sum_{|w|=n}p_{w}\left(s\right)p_{w}\left(t\right)G\left(\check{\varphi}_{w}\left(s\right),\check{\varphi}_{w}\left(t\right)\right).
\]
Equivalently, if $\mathcal{F}_{n}$ denotes the length-$n$ cylinder
$\sigma$-algebra on $\Omega$, then the finite measure on $\Omega$
defined on cylinders by 
\[
\nu^{\left(n\right)}_{s,t}\left(\left[w\right]\right):=p_{w}\left(s\right)p_{w}\left(t\right)
\]
satisfies 
\[
(\widetilde{L}^{n}G)\left(s,t\right)=\int_{\Omega}G\left(\check{\varphi}_{\omega|n}\left(s\right),\check{\varphi}_{\omega|n}\left(t\right)\right)d\nu^{\left(n\right)}_{s,t}\left(\omega\right),
\]
where the integrand is $\mathcal{F}_{n}$-measurable. 
\end{lem}

\begin{proof}
The identity is proved by induction on $n$. The case $n=1$ is the
definition of $\widetilde{L}$, since $p_{i}\left(s\right)=p_{\left(i\right)}\left(s\right)$.
Suppose the identity holds at level $n$ and apply $\widetilde{L}$
once more: 
\[
(\widetilde{L}^{n+1}G)\left(s,t\right)=\sum^{m}_{i=1}p_{i}\left(s\right)p_{i}\left(t\right)(\widetilde{L}^{n}G)\left(\varphi_{i}\left(s\right),\varphi_{i}\left(t\right)\right).
\]
Insert the induction hypothesis at $\varphi_{i}\left(s\right),\varphi_{i}\left(t\right)$
to obtain 
\[
(\widetilde{L}^{n+1}G)\left(s,t\right)=\sum^{m}_{i=1}\sum_{|w|=n}p_{i}\left(s\right)p_{w}\left(\varphi_{i}\left(s\right)\right)p_{i}\left(t\right)p_{w}\left(\varphi_{i}\left(t\right)\right)G\left(\check{\varphi}_{wi}\left(s\right),\check{\varphi}_{wi}\left(t\right)\right).
\]
Using the multiplicative rule $p_{wi}\left(s\right)=p_{w}\left(\varphi_{i}\left(s\right)\right)p_{i}\left(s\right)$,
we identify the coefficient as $p_{wi}\left(s\right)p_{wi}\left(t\right)$
and recover the claimed sum over words of length $n+1$. The integral
form is just a rewriting of the same cylinder expansion. 
\end{proof}

We next connect this boundary expansion to the defect decomposition
from \prettyref{sec:2}. Recall the defect kernels 
\[
D_{n}:=K_{n+1}-K_{n},\qquad n\geq0,
\]
which are positive definite by monotonicity of the kernel tower. Writing
$D_{0}=LK-K$, one has $D_{n}=L^{n}D_{0}$ and hence, by \prettyref{prop:5-3},
\[
D^{\left(h\right)}_{n}=\widetilde{L}^{n}(D^{\left(h\right)}_{0}).
\]
Therefore \prettyref{lem:5-4} yields an explicit boundary expansion
for each normalized defect kernel: 
\[
D^{\left(h\right)}_{n}\left(s,t\right)=\sum_{|w|=n}p_{w}\left(s\right)p_{w}\left(t\right)D^{\left(h\right)}_{0}\left(\check{\varphi}_{w}\left(s\right),\check{\varphi}_{w}\left(t\right)\right).
\]
Summing in $n$ and using the telescoping identity $K_{\infty}=K+\sum_{n\geq0}D_{n}$
from \prettyref{sec:2}, we obtain a boundary organized representation
for the normalized invariant completion: 
\[
K^{\left(h\right)}_{\infty}\left(s,t\right)=K^{\left(h\right)}\left(s,t\right)+\sum^{\infty}_{n=0}\sum_{|w|=n}p_{w}\left(s\right)p_{w}\left(t\right)D^{\left(h\right)}_{0}\left(\check{\varphi}_{w}\left(s\right),\check{\varphi}_{w}\left(t\right)\right),\qquad s,t\in X^{+}_{\mathrm{fin}}.
\]
This formula is purely deterministic and already exhibits the boundary
filtration by word length. We now include a Hilbert space realization
of the same expansion as an $L^{2}$ boundary Gram kernel; this is
the boundary analogue of the direct-sum defect realization from \prettyref{sec:2}.
\begin{thm}
\label{thm:5-5} Let $H_{D^{\left(h\right)}_{0}}$ be the RKHS of
the normalized defect kernel $D^{\left(h\right)}_{0}$ on $X^{+}_{\mathrm{fin}}$.
Define a Hilbert space 
\[
H_{\partial}:=\ell^{2}\left(\mathbb{N}_{0}\right)\otimes L^{2}\left(\Omega,\nu\right)\otimes H_{D^{\left(h\right)}_{0}},
\]
where $\nu$ is any probability measure on $\Omega$ for which cylinders
have positive measure (for example the Bernoulli measure). For each
$s\in X^{+}_{\mathrm{fin}}$ define $\Psi\left(s\right)\in H_{\partial}$
by the rule that its $(n,\omega)$-fiber equals 
\[
\Psi\left(s\right)\left(n,\omega\right)=\frac{p_{\omega|n}\left(s\right)}{\sqrt{\nu\left(\left[\omega|n\right]\right)}}(D^{\left(h\right)}_{0})_{\check{\varphi}_{\omega|n}\left(s\right)}\in H_{D^{\left(h\right)}_{0}},\qquad n\geq0,
\]
where $(D^{\left(h\right)}_{0})_{x}$ denotes the kernel section at
$x$ in $H_{D^{\left(h\right)}_{0}}$. Then for all $s,t\in X^{+}_{\mathrm{fin}}$,
\[
\left\langle \Psi\left(s\right),\Psi\left(t\right)\right\rangle _{H_{\partial}}=\sum^{\infty}_{n=0}D^{\left(h\right)}_{n}\left(s,t\right).
\]
Consequently, 
\[
K^{\left(h\right)}_{\infty}\left(s,t\right)=K^{\left(h\right)}\left(s,t\right)+\left\langle \Psi\left(s\right),\Psi\left(t\right)\right\rangle _{H_{\partial}}.
\]
\end{thm}

\begin{proof}
Fix $s,t\in X^{+}_{\mathrm{fin}}$. By construction and orthogonality
in $\ell^{2}\left(\mathbb{N}_{0}\right)$, 
\[
\left\langle \Psi\left(s\right),\Psi\left(t\right)\right\rangle _{H_{\partial}}=\sum^{\infty}_{n=0}\int_{\Omega}\left\langle \Psi\left(s\right)\left(n,\omega\right),\Psi\left(t\right)\left(n,\omega\right)\right\rangle _{H_{D^{\left(h\right)}_{0}}}\,d\nu\left(\omega\right).
\]
For fixed $n$, the integrand is constant on cylinders $\left[w\right]$
with $|w|=n$, and on such a cylinder we have $\omega|n=w$. Using
$\left\langle (D^{\left(h\right)}_{0})_{x},(D^{\left(h\right)}_{0})_{y}\right\rangle =D^{\left(h\right)}_{0}\left(x,y\right)$,
we obtain 
\begin{align*}
\int_{\Omega}\left\langle \Psi\left(s\right)\left(n,\omega\right),\Psi\left(t\right)\left(n,\omega\right)\right\rangle d\nu\left(\omega\right) & =\sum_{|w|=n}\frac{p_{w}\left(s\right)p_{w}\left(t\right)}{\nu\left(\left[w\right]\right)}D^{\left(h\right)}_{0}\left(\check{\varphi}_{w}\left(s\right),\check{\varphi}_{w}\left(t\right)\right)\nu\left(\left[w\right]\right)\\
 & =\sum_{|w|=n}p_{w}\left(s\right)p_{w}\left(t\right)D^{\left(h\right)}_{0}\left(\check{\varphi}_{w}\left(s\right),\check{\varphi}_{w}\left(t\right)\right).
\end{align*}
By \prettyref{lem:5-4} applied to $G=D^{\left(h\right)}_{0}$, the
right-hand side equals $\left(\widetilde{L}^{n}D^{\left(h\right)}_{0}\right)\left(s,t\right)=D^{\left(h\right)}_{n}\left(s,t\right)$.
Summing over $n$ gives 
\[
\left\langle \Psi\left(s\right),\Psi\left(t\right)\right\rangle _{H_{\partial}}=\sum^{\infty}_{n=0}D^{\left(h\right)}_{n}\left(s,t\right).
\]
Finally, $K^{\left(h\right)}_{\infty}=K^{\left(h\right)}+\sum_{n\geq0}D^{\left(h\right)}_{n}$
is the normalized form of the telescoping identity from \prettyref{sec:2},
and this yields the claimed representation. 
\end{proof}

\begin{rem}
The map $s\mapsto\Psi\left(s\right)$ produces a boundary indexed
family of features built from the one-step defect $D_{0}$ propagated
along the tree, with weights dictated by the Doob path measures. The
$\ell^{2}\left(\mathbb{N}_{0}\right)$ factor records the boundary
filtration by word length, and the orthogonality across levels matches
the defect splitting from \prettyref{sec:2}. In particular, the boundary
representation above is canonical up to the choice of a reference
cylinder-positive measure $\nu$ on $\Omega$, and it is compatible
with the Gaussian model of \prettyref{sec:4} in the sense that it
gives an explicit boundary feature space whose Gram kernel is the
accumulated defect $\sum_{n\geq0}D^{\left(h\right)}_{n}$.

We do not consider boundary limits along $\omega\mapsto s_{n}\left(\omega\right)$
here, but the path measures $\mu_{s}$ and the boundary feature map
$\Psi$ provide a convenient platform for such refinements, including
martingale and boundary compactification questions, as well as for
constructing additional $L$-invariant kernels by altering the defect
contribution level-by-level along $\Omega$. 
\end{rem}

\bibliographystyle{amsalpha}
\bibliography{ref}

\end{document}